\documentclass[11pt]{amsart}
\title{Wall-Crossing in Genus Zero Landau-Ginzburg Theory}
\author{Dustin Ross}
\address{University of Michigan, Department of Mathematics, Ann Arbor, MI 48109, USA}
\email{dustyr@umich.edu}
\author{Yongbin Ruan}
\email{ruan@umich.edu}

\usepackage {color,graphicx,psfrag,verbatim,amssymb,amscd,enumerate,subfigure,mathrsfs,bm}
\usepackage{texdraw,appendix,stackrel}

\newcommand{\proj}{\mathbb{P}}

\newcommand{\Z}{\mathbb{Z}}

\newcommand{\M}{\overline{\mathcal{M}}}
\newcommand{\so}{\mathcal{O}}
\newcommand{\Q}{\mathbb{Q}}
\newcommand{\C}{\mathbb{C}}

\newcommand{\B}{\mathcal{B}}

\newcommand{\mult}{\text{mult}}
\newcommand{\nar}{\text{nar}}
\newcommand{\cH}{\mathcal{H}}
\newcommand{\cF}{\mathcal{F}}
\newcommand{\cL}{\mathcal{L}}
\newcommand{\cJ}{\mathcal{J}}
\newcommand{\cN}{\mathcal{N}}
\newcommand{\cD}{\mathcal{D}}
\newcommand{\bt}{\mathbf{t}}
\newcommand{\bu}{\mathbf{u}}

\newcommand{\bI}{\mathbb{I}}
\newcommand{\btau}{\bm\tau}
\newcommand{\cR}{\mathcal{R}}
\newcommand{\cW}{\mathcal{W}}

\newtheorem{dummy}{}[section]
\newtheorem{dum}{}[]
\newtheorem{lemma}[dummy]{Lemma}

\newtheorem{theorem}[dummy]{Theorem}
\newtheorem{corollary}[dummy]{Corollary}

\newtheorem{theorem1}[dum]{Theorem}
\newtheorem{corollary1}[dum]{Corollary}

\newtheorem*{theorem1'}{Theorem 1'}
\newtheorem*{theorem2'}{Theorem 2'}

\theoremstyle{definition}

\newtheorem{definition}[dummy]{Definition}

\newtheorem{remark}[dummy]{Remark}

\usepackage{graphicx}

\begin{document}

\subjclass[2010]{Primary 14n35; Secondary 53d45.}

\begin{abstract}
We study a family of moduli spaces and corresponding quantum invariants introduced recently by Fan--Jarvis--Ruan. The family has a wall-and-chamber structure relative to a positive rational parameter $\epsilon$. For a Fermat quasi-homogeneous polynomial W (not necessarily Calabi--Yau type), we study natural generating functions packaging the invariants. Our wall-crossing formula relates the generating functions by showing that they all lie on the Lagrangian cone associated to the Fan--Jarvis--Ruan--Witten theory of $W$.  For arbitrarily small $\epsilon$, a specialization of our generating function is a hypergeometric series called the big $I$-function which determines the entire Lagrangian cone. As a special case of our wall-crossing, we obtain a new geometric interpretation of the Landau-Ginzburg mirror theorem.
\end{abstract}

\maketitle

\tableofcontents


\section*{Introduction}\label{sec:intro}
Let $W(x_1,\dots,x_N)$ be a quasi-homogeneous polynomial with weights $(w_1,\dots,w_N)$ and degree $d$:
\begin{equation}\label{polynomial}
W(\lambda^{w_1}x_1,\dots,\lambda^{w_N}x_N)=\lambda^dW(x_1,\dots,x_N).
\end{equation}
We require that $\gcd(w_1,\dots,w_N,d)=1$ and we define the \emph{charges} $q_j:=\frac{w_j}{d}$ and set $q:=\sum q_j$. We also require $W$ to be {\em non-degenerate} in the sense that the charges $q_j$ are uniquely determined from $W$, and we require that the affine variety defined by $W$  is singular only at the origin. Let $X_W$ denote the hypersurface defined by the vanishing of $W$ in weighted projective space $\proj(w_1,\dots,w_N)$.

The Landau-Ginzburg/Calabi--Yau (LG/CY) correspondence asserts the equivalence of two cohomological field theories (CohFTs):
\begin{enumerate}
\item Fan--Jarvis--Ruan--Witten (FJRW) theory of $W$ -- defined via intersection numbers on moduli spaces of curves equipped with $W$-structures, and
\item Gromov--Witten (GW) theory of $X_W$ -- defined via intersection numbers on moduli spaces of stable maps to $X_W$.
\end{enumerate}

The LG/CY correspondence was proved by Chiodo--Iritani--Ruan in genus zero for Calabi--Yau Fermat polynomials (i.e. $W=\sum_j x_j^{d/w_j}$ where $w_j|d$ for all $j$, and $q=1$) \cite{cir:lgcy}. In their proof, they developed a mirror theorem which relates the FJRW theory of $W$ to a hypergeometric series encoding the solutions of the Picard--Fuchs equations on the mirror family near the Gepner point. The correspondence with GW theory follows by analytically continuing the solutions from the Gepner point to the large complex structure limit where the analogous mirror theorem was proved by Givental and Lian--Liu--Yau \cite{g:egwi,lly:mp}.

Rather than use mirror symmetry, Witten has suggested that for any such $W$ there exists a family of A-model CohFTs which interpolates directly between GW and FJRW theory; this is the theory of the \emph{gauged linear sigma model} (GLSM) \cite{w:glsm}. One expects that the LG/CY correspondence can be realized by studying the CohFTs across this family. Since the family has a natural wall-and-chamber structure, this leads to the idea of \emph{wall-crossing} in the A-model. One motivation for the wall-crossing approach, from the perspective of the LG/CY correspondence, is that it may give a better approach to understanding the higher genus comparison proposed by Chiodo--Ruan \cite{cr:lgcyq}.

In fact, Witten's  GLSM applies in much greater generality than hypersurfaces in weighted projective space, and Fan--Jarvis--Ruan have recently provided a mathematically rigorous definition of the general GLSM \cite{fjr:glsm}. In the case of hypersurfaces, the GLSM is a one-dimensional family of CohFTs over the nonzero rational numbers.  Over the positive rationals lies the so-called {\em geometric phase}.  This phase corresponds to a ``quasi-map'' version of stable maps with a field introduced by Chang--Li \cite{cl:gwismwf}. It is expected to be equivalent to the quasi-map theory defined by Ciocan-Fontanine--Kim--Maulik \cite{cfkm:sqtgit}.  Wall-crossing formulas in the geometric phase have been studied in a sequence of papers by Ciocan-Fontanine--Kim \cite{cfk:wc,cfk:wc2,cfk:bif}.


The purpose of this paper is to study the family of CohFTs which arise over the negative rationals corresponding to the {\em Landau-Ginzburg phase}. In particular, for a Fermat polynomial $W$ we prove wall-crossing formulas analogous to those of Ciocan-Fontanine--Kim. Moreover, we show that our wall-crossing formulas recover the LG mirror theorem by making an appropriate specialization, and we give a new enumerative interpretation of the mirror theorem in the sense that we write all of the key players in the mirror theorem in terms of weighted FJRW invariants. 

We now give a more in-depth overview of our results. We refer the reader to Section \ref{sec:main} for precise definitions. 

\subsection*{Statement of Results}

We study a family of moduli spaces $\cR^d_{\vec k,\epsilon\vec l}$ where $\epsilon\in\Q_{>0}$ (for notational convenience, we normalize the GLSM construction so that we work with a positive rational parameter).  These moduli spaces parametrize $\epsilon$-stable pairs $(C,L)$ where $C$ is a Hassett-stable rational orbifold curve with $m$ marked orbifold points $x_1,\dots,x_m$ of weight one and $n$ marked smooth points $y_1,\dots,y_n$ of weight $\epsilon$, and $L$ is an orbifold line bundle which is a $d$-th root of an appropriate twist of the canonical bundle:
\[
L^{\otimes d}\cong \omega_{C,log}\left(-\sum l_iy_i\right):=\omega_{C}\left(\sum x_i-\sum l_iy_i\right)
\]
The vector $\vec k$ records the \emph{multiplicities} of the line bundles at the orbifold marked points.

For a Fermat polynomial, we define a certain vector space $H_W'$, called the \emph{narrow state space}, and for $\phi_i\in H_W'$ we define numerical invariants 
\[
\left\langle \phi_{k_1}\psi^{j_1},\dots,\phi_{k_m}\psi^{j_m} \big | \phi_{l_1},\dots,\phi_{l_n} \right\rangle_{m,n}^{W,\epsilon}\in\Q
\]
 by capping tautological psi classes against the ``Witten class'' in $H_*\left(\cR^d_{\vec k, \epsilon\vec l},\Q\right)$. We package these invariants in generating functions using double bracket notation:
\[
\left\langle\left\langle \phi_{k_1}\psi^{j_1},\dots,\phi_{k_l}\psi^{j_l} \right\rangle\right\rangle_l^{W,\epsilon}(t,u):=\sum_{m,n}\frac{1}{m!n!}\left\langle \phi_{k_1}\psi^{j_1},\dots,\phi_{k_l}\psi^{j_l},\bt(\psi)^m \big | \bu^n \right\rangle_{l+m,n}^{W,\epsilon}
\]
where $\mathbf{t}(z):=\sum t_j^k\phi_kz^j\in H'_W[z]$, $\bu:=\sum u^k\phi_k\in H_W'$, and
\[
\mathbf{ t}(\psi)^m:=\bt(\psi_1),\dots,\bt(\psi_m).
\]  
The sum is over all $m,n\geq 0$ for which the underlying moduli space exists (i.e. we exclude terms where $l+m+n\epsilon\leq 2$). We remark that, in contrast to the double bracket notation typically used in GW theory, our notation encodes arbitrary \emph{descendant} insertions.

The \emph{genus zero descendant potential} is defined by
\begin{equation}\label{potentialfunction}
\cF_W^\epsilon(t,u):=\left\langle\left\langle \hspace{.2cm} \right\rangle\right\rangle_{0}^{W,\epsilon}=\sum_{m,n}\frac{1}{m!n!}\left\langle\bt(\psi)^m \big | \bu^n \right\rangle_{m,n}^{W,\epsilon}.
\end{equation}
For $\epsilon>1$ (denoted $\epsilon=\infty$), the invariants vanish for $n>0$ and we recover the narrow FJRW descendant potential $\cF^\infty(t)$.

For $\epsilon>0$ we package the $t$ derivatives of $\cF^\epsilon$ in the {\em large $\cJ^\epsilon$-function}:
\begin{align}\label{eqn:jfxn}
\nonumber \cJ^\epsilon(t,u,z)&:=z\phi_0\sum_{\substack{a_i\geq 0, i\in\nar \\   \sum a_i \leq\left\lceil \frac{1}{\epsilon} \right\rceil  }}\prod_i\frac{1}{a_i!}\left(\frac{u^i\phi_i}{z}\right)^{a_i}\prod_{j=1}^N\prod_{\substack{0\leq b< \sum_ia_i\langle iq_j\rangle\\ \langle b \rangle =\langle\sum_i a_i iq_j\rangle}}(b+q_j)z\\
& +\mathbf{t}(-z)+\sum_{k} \phi^k\left\langle\left\langle\frac{\phi_k}{z-\psi_{1}} \right\rangle\right\rangle^{\epsilon}_1
\end{align}
where $\langle * \rangle$ denotes the fractional part of $*$. The initial sum in \eqref{eqn:jfxn} corresponds to the \emph{unstable terms} where the underlying moduli spaces in the double bracket do not exist. In Section \ref{sec:equivariant} we give an equivariant interpretation of the terms in the $\cJ$-functions; in particular, we show how both the stable and the unstable terms arise naturally in localization computations on graph spaces. 

Following Givental, we encode the FJRW theory of $W$ in the graph of the differential $\cL:=d\cF^\infty$.  The tautological equations satisfied by $\cF^\infty(t)$ imply that $\cL$ has very special geometric properties.  By definition, $\cJ^\infty(t,-z)$ lies on the cone $\cL$ and our main result extends this property to all $\epsilon$.

\begin{theorem1}[Theorem \ref{thm}]\label{thm1}
Assume $W$ is Fermat. For all $\epsilon>0$, $\cJ^\epsilon(t,u,-z)$ is an $\cH[[u]]$-valued point of $\cL$.
\end{theorem1}

Theorem \ref{thm1} has several important consequences. The first consequence is a direct comparison of large $\cJ$-functions.

\begin{theorem1}[Theorem \ref{cor1}]\label{thm2}
For any $\epsilon_1,\epsilon_2>0$,
\[
\frac{\cJ^{\epsilon_1}(\tau^{\epsilon_1}(t,u),u,z)}{J_0^{\epsilon_1}(u)} = \frac{\cJ^{\epsilon_2}(\tau^{\epsilon_2}(t,u),u,z)}{J_0^{\epsilon_2}(u)}
\]
where $J_0^\epsilon(u)=1+O(u)$ is a scalar function and $\tau^\epsilon(t,u)$ is an invertible change of variables. In particular, since $\cJ$ encodes the $t$ derivatives of $\cF$, we conclude that any genus zero descendant invariant for $\epsilon_1$ which is not constant in $t$ can be recovered from descendant invariants for $\epsilon_2$.
\end{theorem1}

The chamber $\epsilon=\infty$ is important as it recovers FJRW theory. Another important chamber occurs in the limit $\epsilon\rightarrow 0$ (denoted $\epsilon=0$). Define the \emph{big $I$-function} to be the formal series obtained by taking $\epsilon\rightarrow 0$ and $t=0$ in \eqref{eqn:jfxn}:
\[
\bI(u,z)=z\phi_0\sum_{a_i\geq 0}\prod_i\frac{1}{a_i!}\left(\frac{u^i\phi_i}{z}\right)^{a_i}\prod_{j=1}^N\prod_{\substack{0\leq b<\sum_ia_i\langle iq_j\rangle\\ \langle b \rangle =\langle\sum_i a_i iq_j\rangle}}(b+q_j)z.
\]
The next result is obtained from Theorem \ref{thm1} and the special geometry of $\cL$.

\begin{theorem1}[Theorem \ref{cor2}]\label{thm3}
 The big $I$-function $\bI(u,-z)$ is an $\cH[[u]]$-valued point of $\cL$. Moreover, $\bI(u,-z)$ determines $\cL$.
\end{theorem1}

Theorem \ref{thm3} can be viewed as a mirror theorem. In the CY case ($q=1$), Theorem \ref{thm3} was previously proved via the twisted theory formalism by Chiodo--Iritani--Ruan \cite{cir:lgcy}. For $q\neq 1$, it was proved independently by Acosta \cite{a:thesis}.

If we specialize the big $I$-function by restricting to the degree $\leq 1$ part $\underline{\bu}\in H'_{\deg(\phi_i)\leq 1}$, then we obtain the \emph{small $I$-function}:
\[
I(\underline{u},z):=\bI(u,z)|_{u^k=0\text{ if }\deg(\phi_k)>1}.
\] 
As a special case of Theorem \ref{thm1}, we recover the usual LG mirror theorem which relates the small $I$-function to the FJRW $J$-function $J(t_0,z)$ defined by restricting $\cJ^\infty$ to primary variables $t_0=(t_0^k)$.

\begin{corollary1}[Corollary \ref{cor3}]\label{thm4}
\[
I(\underline{u},z)=I_0(\underline{u})z\phi_0+I_1(\underline{u})+O\left(\frac{1}{z}\right)
\]
and 
\[
\frac{I(\underline{u},z)}{I_0(\underline{u})}=J(\eta(\underline{u}),z).
\]
where  $\bm\eta(\underline{u})=\frac{I_1(\underline{u})}{I_0(\underline{u})}$ is the mirror map.
\end{corollary1}

A new aspect of our proof of the mirror theorem is that we obtain an enumerative interpretation of the small $I$-function in terms of invariants at $\epsilon=0$. 

\begin{theorem1}[Theorem \ref{cor4}]\label{thm5}
\[
I(\underline{u},z)=zI_0(\underline{u})\phi_0+z\sum_k \phi^k\left\langle\left\langle I_0(\underline{u})\phi_0,\frac{\phi_k}{z-\psi_2}\right\rangle\right\rangle^0_2(0,\underline u)
\]
Namely, after a correction of $\phi_0$ by $I_0(\underline{u})$ (to account for the failure of the string equation), the small $I$-function is equal to the small $J$-function (at $\epsilon=0$) without any change of variables.
\end{theorem1}

The analog of Theorem \ref{thm5} in the geometric phase was first proved by Cooper--Zinger \cite{cz:msfsqi} for hypersurfaces. Later, Ciocan-Fontanine--Kim \cite{cfk:wc} gave an independent proof of Theorem \ref{thm5} in the geometric phase which applies to a much more general class of targets.

\subsection*{Plan of the Paper}

In Section \ref{sec:main} we introduce weighted FJRW theory in the narrow, Fermat case, generalizing the usual definitions of FJRW theory.  We define the moduli spaces, invariants, and generating functions which appear in Theorem \ref{thm1}.  Our proof of the wall-crossing formula is a two-step process contained in Section \ref{sec:proofs}. In Section \ref{sec:equivariant} we use torus actions on graph spaces to produce relations among the coefficients of the large $\cJ$-functions.  In Section \ref{sec:proof} we give an explicit characterization of points on $\cL_W$ which reduces to the relations developed in Section \ref{sec:equivariant}. In Section \ref{sec:mirror} we deduce Theorems \ref{thm2}, \ref{thm3}, and \ref{thm5} from Theorem \ref{thm1}.

\subsection*{Acknowledgements}

 The first author is grateful to Pedro Acosta, Emily Clader, Nathan Priddis, and Mark Shoemaker for enlightening conversations related to this project. The second author would like to thank Huijun Fan and Tyler Jarvis for their collaboration on gauged linear sigma models on which the current work is based.  He would also like to thank Ionu\c{t} Ciocan-Fontanine and Bumsig Kim for explaining to him their wall-crossing formulas in the geometric phase from which the current work draws a great deal of inspiration. Some of the techniques in this work are motivated by recent work of Coates--Corti--Iritani--Tseng, both authors would like to thank Hsian-Hua Tseng for explaining this work to them.
   The first author has been supported by NSF RTG grants DMS-0943832 and DMS-1045119 and the NSF postdoctoral
research fellowship DMS-1401873. The second author is partially supported by NSF grant DMS-1103368 and NSF FRG grant DMS-1159265.

\section{Weighted Landau--Ginzburg Theory and Wall-Crossing}\label{sec:main}

\subsection{Moduli Spaces}\label{sec:modspaces}

The mathematical theory of GLSMs introduced recently by Fan-Jarvis-Ruan applies to a general situation of GIT-quotients.
The construction greatly simplifies for Fermat hypersurfaces in weighted projective space. For the reader's convenience, we present a self-contained description of the construction in this setting.

\begin{definition}
For $\epsilon\in\Q_{>0}$, a \emph{$(d,\epsilon)$-stable rational curve} is a rational, connected, at worst nodal orbifold curve $C$ along with $m$ distinct orbifold marked points $x_1,\dots,x_m$ and $n$ (not necessarily distinct) smooth marked points $y_1,\dots,y_n$ satisfying
\begin{enumerate}[i.]
\item all nodes and orbifold marks $x_i$ have cyclic isotropy $\mu_d$ and the orbifold structure is trivial away from these points;
\item $\mult_z\left(\epsilon\sum[y_i]\right)\leq 1$ at all points $z\in C$;
\item $\omega_{log}(\epsilon \sum[y_i]):=\omega_C(\sum[x_i]+\epsilon \sum[y_i])$ is ample where $\omega_C$ is the orbifold dualizing sheaf and $[x_i]$ is the orbifold point $\B\mu_d$.
\end{enumerate}
We denote the moduli space of $(d,\epsilon)$-stable curves by $\M^d_{0,m+\epsilon n}$. For $\epsilon>1$, the moduli space $\M^d_{0,m+\epsilon n}$ is empty unless $n=0$.
\end{definition}

By forgetting the orbifold structure, there is a morphism \[\M^d_{0,m+\epsilon n}\rightarrow \M_{0,m+\epsilon n}\] where the latter is a moduli space of Hassett-stable curves \cite{h:mswpsc}.

\begin{definition}\label{def:dspin}
For $\vec l:=(l_1,\dots,l_n)$ with $0\leq l_i\leq d-1$, an \emph{$\vec l$-twisted $d$-spin structure} on a $(d,\epsilon)$-stable curve consists of an orbifold line bundle $L$ and an isomorphism
\[
L^{\otimes d}\stackrel{\kappa}{\longrightarrow}\omega_{log}(-\sum l_i[y_i]).
\]
We denote the moduli space of  $(d,\epsilon)$-stable curves with $\vec l$-twisted $d$-spin structures by $\cR^{d}_{m,\epsilon \vec l}$; it is a smooth Deligne--Mumford stack.
\end{definition}

We recover the usual moduli spaces of $d$-spin curves by specializing $n=0$. The restriction of $L$ to an orbifold point $x$ is nothing more than a character of $\mu_d$, i.e. multiplication by $e^{2\pi i k/d}$ for some $0\leq k<d$. The integer $k$ is called the \emph{multiplicity} of $L$ at $x$ and is denoted $\mult_xL$. The spaces $\cR^{d}_{m,\epsilon \vec l}$ consist of a disjoint union of components indexed by the multiplicities of the line bundle at the orbifold points $x_i$. 

\begin{definition}
For $\vec k=(k_1,\dots,k_m)$ with $0\leq k_i\leq d-1$, let $\cR^d_{\vec k,\epsilon \vec l}$ denote the component in $\cR^d_{m,\epsilon\vec l}$ where $\mult_{x_i}L=k_i+1$ (mod $d$).
\end{definition}


The isomorphism $\kappa$ imposes a non-emptiness condition on the spaces $\cR^d_{\vec k,\epsilon \vec l}$ by comparing degrees of line bundles.  Explicitly, we compute
\begin{equation}\label{eqn:dim1}
\dim\left(\cR^d_{\vec k,\epsilon \vec l}\right)=\begin{cases}
m+n-3 & 2+\sum k_i+\sum l_i\in d\Z\\
-1 & \text{else}
\end{cases}
\end{equation}

Graph spaces $\cR^{G,d}_{\vec{k},\epsilon\vec{l}}$ are defined by enriching the objects in $\cR^d_{\vec k,\epsilon \vec l}$ with a degree one map $f:C\rightarrow\proj^1$. The map to $\proj^1$ has the effect of parametrizing an irreducible component $\hat C\subset C$ and we only require that $\omega_{log}(\epsilon \sum[y_i])$ is ample on $\overline{ C\setminus \hat C}$. The non-emptiness condition on graph spaces is the same as \eqref{eqn:dim1}. When a graph space is nonempty, its dimension is three greater than the dimension of the corresponding non-graph space.

By forgetting the line bundle $L$ and the orbifold structure on $C$, the weighted spin (graph) spaces admit maps to Hassett weighted (graph) spaces:
\[
\theta:\cR^d_{\vec k,\epsilon \vec l}\rightarrow\M_{0,m+\epsilon n} \hspace{.5cm} \text{ and } \hspace{.5cm} \theta:\cR^{G,d}_{\vec{k},\epsilon\vec{l}}\rightarrow\M_{0,m+\epsilon n}(\proj^1,1).
\]  
We define $\psi$-classes on $\cR^{(G,)d}_{\vec k,\epsilon \vec l}$ by $\psi_i:=c_1(\theta^*\mathbb{L}_i)$ where $\mathbb{L}_i$ ($1\leq i \leq m$) is the cotangent line bundle at the $i$-th marked point.

\subsection{The Witten Class and Invariants}\label{sec:virtualclasses}

In this section we define the Witten class and weighted FJRW invariants. We start with a discussion of state spaces.

\subsubsection{State Space}

In the LG/CY correspondence, a certain \emph{state space} related to the FJRW theory of $W$ is identified with the cohomology vector space $H^*(X_W)$. In other words, the FJRW state space indexes the FJRW invariants in the same way that $H^*(X_W)$ indexes the GW invariants. There is a decomposition of the FJRW state space into \emph{narrow} and \emph{broad} sectors that corresponds to the decomposition of $H^*(X_W)$ into ambient and primitive parts.

In the current situation, we define the \emph{extended narrow state space} associated to $W$ to be $H_W:=\Q^{d}$ with basis vectors denoted by $\phi_k$ with $0\leq k \leq d-1$. In the definition of the unstable terms of $\cJ$, we use a multiplication on $H_W$ defined by
\[
\phi_i\cdot\phi_j:=\phi_{i+j \text{ mod } d}
\] 
The reader should not confuse this product with the quantum ring structure of FJRW theory.

The \emph{narrow sector} $H_W'=\bigoplus_{k\in \nar}\Q_{\phi_k}$ is obtained by restricting to the vectors indexed by the set
\[
\nar:=\{k: \text{for all } j, \langle q_j (k+1) \rangle\neq 0\}.
\]
We refer to the complement of the narrow sector in the extended narrow state space as the \emph{broad sector} (the astute reader will notice that our `broad sector' is only a small part of the broad sector defined in \cite{fjr:we}). There is a perfect pairing on the narrow sector defined by
\[
(\phi_i,\phi_j)_W:=\delta_{i+j,d-2}.
\]  
We set $\phi^i:=\phi_{d-2-i}$.

Narrow states have a geometric interpretation in terms of $\cR_{\vec k, \epsilon\vec l}^d$. Specifically, the condition $k_i\in\nar$ is equivalent to the condition that $\mult_{x_i}L^{w_j}\neq 0$ for all $j$. This observation leads to the next definition.

\begin{definition}
We say an orbifold node $z$ is \emph{narrow} if $\mult_zL^{w_j}\neq 0$ for all $j$ and we call it \emph{broad} otherwise.
\end{definition}

\subsubsection{Witten Class}

The foundations of FJRW theory rest upon the construction of the \emph{Witten class}, a homology class which plays a role analogous to that of the virtual fundamental class in GW theory. The general construction of the Witten class was carried out by Fan--Jarvis--Ruan \cite{fjr:we} and equivalent contructions in the narrow sector have been developed by Polishchuk--Vaintrob \cite{pv:wc}, Chiodo \cite{c:wc}, and Chang--Li--Li \cite{cll:wc}. We only require a very special case of their constructions.

From this point on, assume $W$ is Fermat so that $w_j|d$ for all $j$. With this assumption, the genus zero theory simplifies significantly on the narrow sector due to the fact that one can write the Witten class as the Poincar\'e dual of the top Chern class of an appropriate vector bundle. In the usual FJRW setting (i.e. when $n=0$), the vector bundle is given by
\[
W_{\vec k}:=\bigoplus_j\left(R^1\pi_* L^{w_j}\right)^\vee
\]
where $\pi:C\rightarrow \cR^d_{\vec k}$ is the universal curve (cf. \cite{fjr:we}, Theorem 4.1.8(5)(a)). Notice that the rank of $W_{\vec k}$ can be computed by orbifold Riemann-Roch and is equal to
\[
-N+2q+\sum_{i}\deg\left(\phi_{k_i}\right).
\]
where
\[
\deg(\phi_k):=\sum_j\left\langle q_jk \right\rangle.
\]

In order to define the Witten class more generally for weighted FJRW invariants, we begin by defining integers $s_{i,j}$ and $l_{i,j}$ by
\begin{equation*}\label{eqn:conv}
l_i=:s_{i,j}\frac{d}{w_j}+l_{i,j} \text{ for some } 0\leq l_{i,j}<\frac{d}{w_j}.
\end{equation*}
In weighted FJRW theory, the analog of $L^{w_j}$ is played by
\begin{equation}\label{eqn:conv}
L_j:=L^{w_j}\otimes\so\left(\sum_i s_{i,j}[y_i]\right).
\end{equation}
We have the following vanishing result on cohomology.

\begin{lemma}[Concavity]\label{concavity}
For $\vec k=(k_1,\dots,k_m)$ with $k_i\in\nar$ for all $i$,
\begin{equation}
H^0\left(C,L_j\right)=0
\end{equation}
for all points of $\cR^d_{\vec k,\epsilon \vec l}$.
\end{lemma}

\begin{proof}
Let $Z\subset C$ be an irreducible component and denote the restriction of $L_j$ to $Z$ by $L_j(Z)$. By definition, 
\[
L_j(Z)^{\otimes \frac{d}{w_j}}\cong \left(\omega_{log}\left(-\sum_i l_{i,j}[y_i])\right)\right)\Bigg |_Z.
\] 
Since $\deg(\omega_{log}|Z)=-2+|m_Z|+\#\{Z\cap\overline{C\setminus Z}\}$ where $m_Z\subseteq\{1,\dots,m\}$ is the set of indices of points $x_1,\dots,x_m$ which lie on $Z$, we have
\[
\deg\left(L_j(Z)^{\otimes \frac{d}{w_j}}\right)\leq-2+|m_Z|+\#\{Z\cap\overline{C\setminus Z}\}.
\]

Let $|L_j(Z)|$ denote the push-forward of $L_j(Z)$ to the coarse curve $|Z|$. We obtain
\[
\deg\left(|L_j(Z)|^{\otimes \frac{d}{w_j}}\right)\leq-2+|m_Z|-\sum_{i\in m_Z} \frac{d}{w_j}\left\langle \frac{w_j(k_i+1)}{d} \right\rangle+\#\{Z\cap\overline{C\setminus Z}\}.
\]
Since $k_i\in\nar$ and $w_j|d$, $\frac{d}{w_j}\left\langle \frac{w_j(k_i+1)}{d} \right\rangle\geq 1$ for all $i$. Therefore,
\begin{equation}\label{eqn:sectioninequality}
\deg\left(|L_j(Z)|^{\otimes \frac{d}{w_j}}\right)<\#\{Z\cap\overline{C\setminus Z}\}-1.
\end{equation}
In particular, by induction on the number of components of $|C|$, this implies that $H^0\left(|C|,|L_j|^{\otimes \frac{d}{w_j}}\right)=0$. The lemma follows from the fact that
\[
H^0\left(C,L_j\right)=H^0\left(|C|,|L_j|\right)\rightarrow H^0\left(|C|,|L_j|^{\otimes \frac{d}{w_j}}\right)
\]
where the map $s\rightarrow s^{\frac{d}{w_j}}$ sends nonzero sections to nonzero sections.
\end{proof}

\begin{remark}\label{concavityrmk}
If one of the $k_i$ is broad, then the inequality \eqref{eqn:sectioninequality} is strict for all but (at most) one component and it still follows by induction that $H^0\left(|C|,|L_j|^{\otimes \frac{d}{w_j}}\right)=0$. Therefore, Lemma \ref{concavity} holds even if all but one of the $k_i$ are narrow.
\end{remark}

Due to concavity, $R^1\pi_*L_j$ is a vector bundle when $k_i\in\nar$.  We define a corresponding vector bundle
\[
W_{\vec k, \epsilon\vec l}:=\bigoplus  \left( R^1\pi_*L_j\right)^\vee
\]
and we define the weighted Witten class in the narrow sector by
\begin{equation}\label{virclass}
\left[\cW_{\vec k,\epsilon\vec l}\right]:=e\left(W_{\vec k, \epsilon\vec l}\right)^{\text{PD}}\in H_*\left(\cR^d_{\vec k,\epsilon \vec l},\Q\right)
\end{equation}
where $e(-)$ is the Euler class and $*^\text{PD}$ denotes the Poincar\'e dual of $*$. The (complex) homological degree of $\left[W_{\vec k,\epsilon\vec l}\right]$ can be computed by orbifold Riemann-Roch:
\begin{equation}\label{eqn:dim}
\deg_{\C}\left(\left[\cW_{\vec k,\epsilon\vec l}\right] \right)=N-3-2q+m-\sum_{i}\deg(\phi_{k_i})+n-\sum_{i}\deg(\phi_{l_i}).
\end{equation}

\begin{remark}
The symmetry between $\vec k$ and $\vec l$ in \eqref{eqn:dim} can be viewed as motivation for defining $L_j$ as in \eqref{eqn:conv}.
\end{remark}

We have the following useful vanishing property.

\begin{lemma}[Ramond Vanishing]\label{ramond}
If $\cD$ is a boundary divisor in $\cR^d_{\vec k,\epsilon \vec l}$ with a broad node, 
\[
\left[\cW_{\vec k,\epsilon\vec l}\right]\cap \cD=0.
\]
\end{lemma}

\begin{proof}
Suppose that the broad node $z$ has multiplicity $k$. Since it is broad, $\frac{w_jk}{d}$ is an integer for some $w_j$. Letting $C_1$ and $C_2$ denote the components of $C$ separated by $z$, we have a short exact sequence on $\cD$:
\[
0\rightarrow \pi_*L_j|_z\rightarrow R^1\pi_*L_j \rightarrow R^1\pi_*L_j|_{C_1}\oplus R^1\pi_*L_j|_{C_2}\rightarrow 0
\]
where the vanishing of the initial terms follows from concavity and Remark \ref{concavityrmk}. The nonvanishing of $\pi_*L_j|_z$ is a consequence of broadness. We compute
\[
\pi_*\left(L_j|_z^{\otimes\frac{d}{w_j}}\right)\cong \pi_*\left(\omega_{C}|_{z}\right).
\]
The latter is trivial. Therefore,
$c_1(\pi_*L_j|_z)=0$ which implies 
\[
e\left(\bigoplus_j R^1\pi_*  L_j \right)\cup \cD=0
\]
from which the result follows.
\end{proof}

\begin{remark}
Concavity fails in higher genus, even for $W$ Fermat. However, an algebraic Witten class can still be defined in the narrow sector and Ramond vanishing continues to hold in higher genus if $W$ is Fermat \cite{fjr:glsm}.
\end{remark}

Numerical invariants are defined by integrating $\psi$-classes over the Witten class:
\begin{equation}\label{fjrwinvts}
\left\langle \phi_{k_1}\psi^{j_1},\dots,\phi_{k_m}\psi^{j_m} \big| \phi_{l_1},\dots,\phi_{l_n}  \right\rangle^{W,\epsilon}_{m,n}:=d\int_{\left[\cW_{\vec k,\epsilon\vec l}\right]}\prod\psi_i^{j_i}\in\Q.
\end{equation}
For the weighted points $y_i$, we use $\phi_i$ as a convenient book-keeping device, but the reader should not confuse it with insertions at the orbifold points. The invariants are defined to vanish if any of the $k_i,l_i$ are broad or if the underlying moduli space does not exist.

As in the introduction, we define
\[
\left\langle\left\langle \phi_{k_1}\psi^{j_1},\dots,\phi_{k_l}\psi^{j_l} \right\rangle\right\rangle_l^{W,\epsilon}(t,u):=\sum_{m,n}\frac{1}{m!n!}\left\langle \phi_{k_1}\psi^{j_1},\dots,\phi_{k_l}\psi^{j_l},\bt(\psi)^m \big | \bu^n \right\rangle_{l+m,n}^{W,\epsilon}
\]
where $\mathbf{t}(z):=\sum t_j^k\phi_kz^j\in H'_W[z]$, $\bu:=\sum u^k\phi_k\in H_W'$, and
\[
\bt(\psi)^m:=\bt(\psi_1),\dots,\bt(\psi_m).
\]

The \emph{genus zero descendant potential} is
\[
\cF_W^\epsilon(t,u):=\left\langle\left\langle \hspace{.2cm} \right\rangle\right\rangle_{0}^{W,\epsilon}=\sum_{m,n}\frac{1}{m!n!}\left\langle\bt(\psi)^m \big | \bu^n \right\rangle_{m,n}^{W,\epsilon}.
\]

Graph space invariants will be important in our proof of the wall-crossing formula. Concavity and Ramond vanishing continue to hold for graph spaces $\cR^{G}_{\vec k,\epsilon\vec l}$. We define a Witten class in graph spaces as in \eqref{virclass}. It has homological degree 
\[
\deg_{\C}\left(\left[\cW^{G}_{\vec k,\epsilon\vec l}\right] \right)=N-2q+m-\sum_{i}\deg(\phi_{k_i})+n-\sum_{i}\deg(\phi_{l_i}).
\]
This allows us to define the graph space invariants 
\[
\left\langle \tau_{j_1}(\phi_{k_1}),\dots,\tau_{j_m}(\phi_{k_m})\big |\phi_{l_1},\dots,\phi_{l_n}  \right\rangle^{G,W,\epsilon}_{m,n}
\]
as in \eqref{fjrwinvts}.

\subsection{Givental's Symplectic Formalism}

Following Givental, we define the infinite dimensional vector space $\cH:=H_W'((z^{-1}))$ with symplectic form
\[
\Omega(f(z),g(z)):=\text{Res}_{z=0}(f(z),g(-z))_W. 
\]
$\cH$ has a natural polarization $\cH=\cH^+\oplus\cH^-$ where elements of $\cH^+$ are polynomial in $z$ while elements of $\cH^-$ are formal series with only negative powers of $z$. The natural Darboux coordinates for this polarization are
\[
\mathbf{q}=\sum q_j^k\phi_kz^j\in \cH^+\hspace{1cm}\text{and}\hspace{1cm}\mathbf{p}=\sum p_{k,j}\phi^k(-z)^{-j-1}\in \cH^-.
\]
The Darboux coordinates canonically identify $\cH$ with the cotangent space on $\cH^+$.
 
Via the \emph{dilaton shift}
\[
\mathbf{q}(z):=\mathbf{ t}(z)-\phi_0z,
\]
we can think of $\cF^\infty( t)$ as a formal function on $\cH^+$. We define $\cL$ to be the graph of the differential of $\cF^\infty( t)$ in $\cH$:
\[
\cL:=\{( \mathbf{q, p})\in\cH: \mathbf{p}=\partial_{ \mathbf{q}}\cF^\infty( t)\}.
\]

Givental showed that the tautological equations satisfied by $\cF^\infty$ are equivalent to the fact that $\cL$ is the formal germ of a Lagrangian cone with vertex at the origin such that each tangent space $T$ to the cone is tangent to the cone exactly along $zT$ \cite{g:sgofs}.

\subsection{Wall-Crossing Formula}

Our wall crossing formula relates the {\em large $\cJ^\epsilon$-functions} which package the $t$ derivatives of $\cF^\epsilon$:
\begin{align}\label{jxn2}
\nonumber \cJ^\epsilon(t,u,z)&:=z\phi_0\sum_{\substack{a_i\geq 0, i\in\nar \\   \sum a_i \leq\left\lceil \frac{1}{\epsilon} \right\rceil  }}\prod_i\frac{1}{a_i!}\left(\frac{u^i\phi_i}{z}\right)^{a_i}\prod_{j=1}^N\prod_{\substack{0\leq b<\sum_ia_i\langle iq_j\rangle\\ \langle b \rangle =\langle \sum_i a_i iq_j\rangle}}(b+q_j)z\\
& +\mathbf{t}(-z)+\sum_{k} \phi^k\left\langle\left\langle\frac{\phi_k}{z-\psi_{1}} \right\rangle\right\rangle^{\epsilon}_1
\end{align}
where $\langle * \rangle$ denotes the fractional part of $*$. When $\epsilon>1$, we have
\[
\cJ^\infty(t,-z)=-z\phi_0+\mathbf{t}(z)+\sum_{k}  \phi^k\left\langle\left\langle \frac{\phi_k}{-z-\psi_{m+1}} \right\rangle\right\rangle^{\infty}_{1}(t)
\]
which (by definition) spans the entire cone $\cL$. This leads to the following definition.

\begin{definition} 
By an \emph{$\cH[[u]]$-valued point of $\cL$}, we mean a formal series in $H'[[t,u]]((z^{-1}))$ which is of the form
\begin{equation}\label{eqn:genform}
-z\phi_0+\hat\bt(z)+\sum_{k} \phi^k\left\langle\left\langle\frac{\phi_k}{-z-\psi_{m+1}} \right\rangle\right\rangle^{\infty}_{1}\left(\hat t\right)
\end{equation}
for some $\hat\bt(z)=\bt(z)+O(u)\in \cH^+[[u]]$. 
\end{definition}

Our main result is the following.

\begin{theorem}\label{thm}
For all $\epsilon>0$, $\cJ^\epsilon(t,u,-z)$ is an $\cH[[u]]$-valued point of $\cL$.
\end{theorem}

\section{Proof of Wall-Crossing}\label{sec:proofs}

\subsection{$\cJ$-function Relations}\label{sec:equivariant}

In this section we derive a set of universal relations satisfied by the coefficients of the $\cJ$-functions.  In the process, we give an equivariant description of the $\cJ$-functions in terms of fixed-point contributions to certain auxiliary integrals on graph spaces.

To ease notation, we omit $W$ when confusion does not arise.

\begin{lemma}\label{lem:relation1}
For every $\epsilon>0$, the series 
\begin{equation}\label{eqn:pairing}
\left( \partial_{u^r} \cJ^\epsilon(t,u,z), \partial_{t_0^s} \cJ^\epsilon(t,u,-z) \right)
\end{equation}
is regular at $z=0$ for all narrow $r,s$.
\end{lemma}

\begin{proof}

We prove the lemma via Atiyah-Bott localization on graph spaces, we proceed in several steps.

\subsubsection*{Step 1: Equivariant Setup} 

Let $\C^*$ act on $\proj^1$ via $\lambda[z_0:z_1]:=[\lambda z_0: z_1]$ so that the tangent bundle $T\proj^1$ has weights $1$ at $0=[1:0]$ and $-1$ at $\infty=[0:1]$. Define the equivariant class $[0]:=c_1(\so(1))\in H_{\C^*}^*(\proj^1)$ where $\so(1)$ is linearized with weights $1$ at $0$ and $0$ at $\infty$. Similarly define $[\infty]$ by linearizing $\so(1)$ with weights $0$ at $0$ and $-1$ at $\infty$.  

The $\C^*$ action on $\proj^1$ naturally induces an action on $\cR^{G,d}_{\vec k,\epsilon \vec l}$ and it lifts canonically to an action on the vector bundle $W_{\vec k,\epsilon\vec l}^G$. More specifically, $\C^*$ acts on $C$ by acting on (the coarse coordinates of) the parametrized component $\hat C$ and this action lifts canonically to any space of sections by pre-composing each section with the $\C^*$ action on $C$. This lift defines a canonical equivariant Witten class on graph space: it is the equivariant Euler class of $W_{\vec k,\epsilon\vec l}^G$.

There are equivariant evaluation maps 
\[
ev_i,\tilde{ev}_j:\cR^{G,d}_{\vec k,\epsilon \vec l}\rightarrow \proj^1
\]
which record the images of $x_i,y_j$, respectively, under the map $f$.

For any cohomology class $\gamma\in H^*\left(\cR^{G,d}_{\vec k,\epsilon \vec l}\right)$, we define invariants $\langle\alpha\big | \beta\big | \gamma\rangle^{G,\epsilon}_{m,n}$ by cupping the integrand in the definition of $\langle \alpha\big | \beta \rangle^{G,\epsilon}_{m,n}$ with $\gamma$. In the presence of a torus action, integration is the equivariant push-forward to a point.

\subsubsection*{Step 2: Equivariant Interpretation of \eqref{eqn:pairing}}

Consider the equivariant series
\begin{equation}\label{eqn:graphseries}
\sum_{m,n\geq 0}\frac{1}{m!n!}\left\langle\bt(-\psi)^m,\phi_s\big | \bu^n,\phi_r \big | ev_{m+1}^*([\infty])\cup\tilde{ev}_{n+1}^*([0])\right\rangle^{G,\epsilon}_{m+1,n+1}.
\end{equation}
By definition, \eqref{eqn:graphseries} is regular at $z=0$. Therefore, Lemma \ref{lem:relation1} follows from the claim that \eqref{eqn:pairing}=\eqref{eqn:graphseries}. We prove this claim in Steps 3-5.

\subsubsection*{Step 3: Localization Formula}

By the localization theorem, \eqref{eqn:graphseries} can be computed as a sum of contributions from each $\C^*$ fixed locus.  The fixed loci parametrize curves where all of the marked points and nodes are mapped to $0$ and $\infty$ via $f$. We discard the loci where $f(x_{m+1})=0$ or $f(y_{n+1})=\infty$ because the integrand vanishes when restricted to these loci. For each $\vec k, \vec l$, we obtain a fixed locus for each splitting of the remaining $m+n$ points over $0$ and $\infty$.  Denote such a fixed locus by 
\[
\iota:F_{\vec k_0,\vec l_0}^{\vec k_\infty, \vec l_\infty}\hookrightarrow \cR^{G,d}_{(\vec k,s),\epsilon(\vec l,r)}
\]
where $\vec k_0,\vec k_\infty$ is a splitting of the vector $\vec k$ into subvectors of lengths $m_0,m_\infty$, similar for $\vec l$.  By the Atiyah-Bott localization formula, the series \eqref{eqn:graphseries} is equal to
\begin{equation}\label{eqn:graphseries2}
\sum_F\frac{1}{m!n!}\int_F\frac{\iota^*\left(e\left(W^G\right)\cup\bt(-\psi)^{\vec k}\cup \bu^{\vec l}\cup ev_{m+1}^*([\infty])\cup\tilde{ev}_{n+1}^*([0]) \right)}{e(\cN_F)}
\end{equation}
where the denominator is the equivariant Euler class of the normal bundle and 
\[
\bt(-\psi)^{\vec k}=\bt(-\psi_1)_{k_1},\dots,\bt(-\psi_m)_{k_m}
\] 
with $\bt(z)_k$ the coefficient of $\phi_k$ in $\bt(z)$, similar for $\bu^{\vec l}$.

\subsubsection*{Step 4: Stable Terms}

Define a fixed locus to be stable if it has a node over both $0$ and $\infty$. For stable fixed loci, we have 
\begin{equation}\label{fixedident}
F_{\vec k_0,\vec l_0}^{\vec k_\infty, \vec l_\infty}\cong \cR^{d}_{(\vec k_0,k),\epsilon(\vec l_0,r)}\times \cR^{d}_{(\vec k_\infty,s,d-2-k),\epsilon \vec l_\infty}
\end{equation}
where $k$ is uniquely determined from $\vec k_0$, $\vec l_0$, and the non-emptyness condition \eqref{eqn:dim1}. From the normalization sequence of the curve, we have a long exact sequence in cohomology
\begin{align}\label{longexact}
0&\rightarrow \bigoplus_j \pi_*L_j|_{\hat C}\rightarrow\bigoplus_j\pi_*L_j|_0\oplus\pi_*L_j|_\infty\rightarrow\\
\nonumber&\rightarrow \iota^*\left(W^G\right)^\vee \rightarrow W_0^\vee \oplus W_\infty^\vee\oplus \bigoplus_j R^1\pi_*L_j|_{\hat C}\rightarrow 0
\end{align}
When $k\notin\nar$, the first line is nonzero. The first term is canonically identified with either one of the summands in the second term (we are using the fact that $F$ is a stable locus which implies that $\deg(L_j|_{\hat C})=0$). The remaining summand in the second term has trivial equivariant Euler class by an argument analogous to that in the proof of Lemma \ref{ramond} along with the fact that the canonical $\C^*$ action on this term is trivial. Therefore, $e\left(\iota^*\left(W^G\right)\right)=0$ when $k\notin\nar$.

When $k\in\nar$, the first line of \eqref{longexact} vanishes. Since $\deg(L_j|_{\hat C})=0$, the third summand in the final term also vanishes. Altogether, the above observations imply the following.
\begin{itemize}
\item $ \iota^*e\left(W^G\right)=\begin{cases}e(W_0)e(W_\infty) & k\in\nar\\ 0 & k\notin\nar.\end{cases}$
\end{itemize}
Since the canonical $\C^*$ action on $W_0$ and $W_\infty$ is trivial, this implies in particular that the equivariant Witten class on the graph space restricts to the product of the usual Witten classes on the stable fixed loci. 

We also compute 
\begin{itemize}
\item $\iota^*\psi =\psi$, 
\item $\iota^* ev_{m+1}^*([\infty])=-z$, and
\item $\iota^* \tilde{ev}_{n+1}^*([0])=z$.
\end{itemize}

In addition, the normal bundle in the denominator of \eqref{eqn:graphseries2} contributes factors
\begin{itemize}
\item $\frac{1}{d}(z-\psi_{m_0+1})$ and $\frac{1}{d}(-z-\psi_{m_\infty+1})$ from smoothing the nodes at $0$ and $\infty$ on the \emph{orbifold} curve, and
\item $-z^2$ from deforming the map to $\proj^1$
\end{itemize}

Pulling everything together, we compute that for stable fixed loci the integral in \eqref{eqn:graphseries2} is equal to
\begin{align}\label{eqn:graphseries3}
&\left\langle \bt(-\psi)^{\vec k_0},\frac{\phi_k}{z-\psi_{m_0+1}}\bigg | \bu^{\vec l_0},\phi_r \right\rangle^{\epsilon}_{m_0+1,n_0+1}\\
&\nonumber\hspace{1.5cm}\cdot\left\langle \bt(-\psi)^{\vec k_\infty},\phi_s,\frac{\phi_{d-2-k}}{-z-\psi_{m_\infty+2}}\bigg |\bu^{\vec l_\infty} \right\rangle^{\epsilon}_{m_\infty+2,n_\infty}
\end{align}

Notice that the first factor in \eqref{eqn:graphseries3} captures the stable contributions to the coefficient of $\phi^k$ in $\partial_{u^r} \cJ^\epsilon(t,u,z)$ and the second factor captures the stable contribution to the coefficient of $\phi_k$ in $\partial_{t_0^s} \cJ^\epsilon(t,u,-z)$. 

\subsubsection*{Step 5: Unstable Terms}

There are two cases, either $m_\infty=n_\infty=0$ or $m_0=0$ and $n_0+1\leq \frac{1}{\epsilon}$.

In the first case, the fixed locus has a single $\phi_s$ point over $\infty$. The localization contribution in this case is obtained from \eqref{eqn:graphseries3} by setting $k=s$ and replacing the right hand term with $1$ which is the unstable contribution to the coefficient of $\phi_{k=s}$ in $\partial_{t_0^s} \cJ^\epsilon(t,u,-z)$.

In the second case, the fixed locus has $n_0+1$ light points stacked up at $0\in\hat C$. The long exact sequence \eqref{longexact} becomes
\begin{align}\label{longexact2}
0&\rightarrow \bigoplus_j \pi_*L_j|_{\hat C}\rightarrow\bigoplus_j\left(\pi_*L_j|_\infty\right)\rightarrow\\
\nonumber&\rightarrow \iota^*\left(W^G\right)^\vee \rightarrow W_\infty^\vee\oplus \bigoplus_j R^1\pi_*L_j|_{\hat C}\rightarrow 0
\end{align}

$L_j|_{\hat C}$ is now a negative bundle for all $j$ so the inital term in \eqref{longexact2} vanishes and we compute
\[
\iota^*e\left(W^G\right)=e(W_\infty)e\left( \bigoplus_j(\pi_*L_j|_{\infty})^\vee\oplus (R^1\pi_*L_j|_{\hat C})^\vee\right).
\]
If $r+\sum(\vec l_0)_i\notin\nar$, then for some $j$, $\pi_*L_j|_{\infty}$ has trivial equivariant Euler class as before and therefore $\iota^*e\left(W^G\right)=0$. If $r+\sum(\vec l_0)_i\in\nar$, then $\pi_*L_j|_{\infty}=0$ for all $j$ and 
\[
\iota^*e\left(W^G\right)=e(W_\infty)e\left( \bigoplus_j (R^1\pi_*L_j |_{\hat C})^\vee\right).
\]

Notice that $\hat C\cong \proj(1,d)$ and
\[
L_j|_{\hat C}^{\otimes\frac{d}{w_j}}\cong\omega_{\hat C, log}\left(-r_j-\sum_i (\vec l_0)_{i,j} \right)\cong\so\left(-1-r_j-\sum_i (\vec l_0)_{i,j} \right)
\]
where $r_j:= r$ mod $\frac{d}{w_j}$. In particular, we have
\[
L_j|_{\hat C}\cong \so\left( -\left( q_j+\langle q_j r \rangle+\sum_i\langle q_j(\vec l_0)_i \rangle \right)[\infty]  \right)
\]
and \u{C}ech representatives of cohomology are given by
\[
H^1(\hat C,L_j|_{\hat C})=\left\langle \frac{1}{x_0^{db}x_1^{q_j+\langle q_j r \rangle+\sum_i\langle q_j(\vec l_0)_i \rangle-b}}\Bigg|\substack{0< b<q_j+\langle q_j r \rangle+\sum_i\langle q_j(\vec l_0)_i \rangle \\ \langle b \rangle = \langle q_j+q_j r+\sum_i q_j(\vec l_0)_i \rangle} \right\rangle
\]
where $x_0$, $x_1$ are the orbifold coordinates on $\hat C$, related to the coarse coordinates on $\proj^1$ by $z_0=x_0^d$, $z_1=x_1$. Therefore, we compute
\[
e\left( \bigoplus_j (R^1\pi_*L_j |_{\hat C})^\vee\right)=\prod_{j=1}^N \prod_{\substack{0\leq b<q_j+\langle q_j r \rangle+\sum_i\langle q_j(\vec l_0)_i \rangle \\ \langle b \rangle = \langle q_j+q_j r+\sum_i q_j(\vec l_0)_i \rangle}}bz.
\]

The normal bundle on this fixed locus contains a summand corresponding to deforming the points away from $0$; this contributes an equivariant factor of $z^{n^0+1}$. Pulling everything together, we see that the contributions from the second case of unstable loci is obtained from \eqref{eqn:graphseries3} by setting 
\[
k=d-2-r-\sum (\vec l_0)_i \text{ mod } d
\]
and replacing the left side with 
\[
\frac{1}{z^{n_0}}\prod_{j=1}^N \prod_{\substack{0\leq b<\langle q_j r \rangle+\sum_i\langle q_j(\vec l_0)_i \rangle \\ \langle b \rangle = \langle q_j r +\sum_i q_j(\vec l_0)_i \rangle}}(b+q_j)z
\]
which corresponds to the unstable contribution to the coefficient of $\phi^k$ in $\partial_{u^r}\cJ^\epsilon(t,u,z)$.

Adding \eqref{eqn:graphseries3} over all stable and unstable loci proves the claim.
\end{proof}

\subsection{Cone Characterization}\label{sec:proof}

In this section we prove a characterization of the $\cH[[u]]$ valued points of $\cL$. This characterization along with Lemma \ref{lem:relation1} implies that the $\cJ$-functions lie on $\cL$.

\begin{lemma}\label{thm:determine}
Suppose $F(t,u,z)\in H'[[t,u]]((z^{-1}))$ has the form 
\begin{equation}\label{eqn:form}
F(t,u,z)=z\phi_0+\bt(-z)+f(u,-z)+\underline F(t,u,z)
\end{equation}
where
\begin{enumerate}[(i)]
\item $f(u,z)\in H'[[u]][z]$ with $f(0,z)=0$,
\item $\underline F(t,u,z)\in H'[[t,u,z^{-1}]]$ only has terms of degree $\geq 2$ in $t,u$, and
\item $F(t,u=0,-z)\in\cL$.
\end{enumerate}
Then $F(t,u,-z)\in\cL$ if and only if the series 
\begin{equation}\label{eqn:master}
\left( \partial_{u^r} F(t,u,z), \partial_{t_{0}^s} F(t,u,-z) \right)
\end{equation}
is regular at $z=0$ for all $r,s$.
\end{lemma}

\begin{proof}
First suppose $F(t,u,z)$ satisfies (i)-(iii) and $F(t,u,-z)$ lies on $\cL$. We show that \eqref{eqn:master} is regular at $z=0$.  By definition, $F$ has the form
\[
F(t,u,z)=z\phi_0+\hat\bt(-z)+\sum_{k}\phi^k\left\langle\left\langle \frac{\phi_k}{z-\psi_1}\right\rangle\right\rangle_{1}^{\infty}\left(\hat t \right).
\]
where $\hat\bt(z)=\bt(z)+f(u,z)$. We have
\[
\partial_{u^r} F(t,u,z)=\partial_{u^r} \hat\bt(-z)+\sum_{k}\phi^k\left\langle\left\langle \partial_{u^r} \hat\bt(\psi_1),\frac{\phi_k}{z-\psi_2} \right\rangle\right\rangle_{2}^{\infty}\left(\hat t\right)
\]
and
\[
\partial_{t_{0}^s} F(t,u,-z)=\phi_s+\sum_{k}\phi^k\left\langle\left\langle \phi_s ,\frac{\phi_k}{-z-\psi_2}\right\rangle\right\rangle_{2}^{\infty}\left(\hat t\right).
\]

Proceeding exactly as in the proof of Lemma \ref{lem:relation1}, we see that the series \eqref{eqn:master} is equal to the equivariant series
\begin{equation*}\label{eqseries}
\left\langle\left\langle \partial_{u^r}\hat\bt(\psi_{1}),\phi_s \big | ev_{1}^*([0])\cup ev_{2}^*([\infty])\right\rangle\right\rangle^{G,\infty}_{2}
\end{equation*}
which is regular at $z=0$ by definition.

Now suppose $F$ has the form \eqref{eqn:form} and satisfies (i)-(iii) and \eqref{eqn:master}, we show that $F(t,u,-z)$ lies on $\cL$.  To this end, we show that $F$ is uniquely determined from its regular part, its restriction to $u=0$, and the recursions \eqref{eqn:master}. To see this, write
\[
\underline F=\sum f_{\vec m,\vec n,j,s}\frac{t^{\vec m}u^{\vec n}}{z^j}\phi^s
\]
where $t^{\vec m}=\sum (t_j^k)^{m_j^k}$ and $u^{\vec n}=\sum (u^k)^{n^k}$. Then \eqref{eqn:master} determines the coefficients by induction on $(|\vec n|,|\vec m|)$.  Indeed, assume we know $ f_{\vec m',\vec n',j',s'}$ for all $(|\vec n'|,|\vec m'|)\leq(|\vec n|,|\vec m|)$ and suppose we want to compute the coefficient $ f_{\vec m,(\vec n,r),j,s}$. We consider the relation 
\[
\left( \partial_{u^r} F(t,u,z), \partial_{t_0^s} F(t,u,-z) \right)\left[\frac{t^{\vec m}u^{\vec n}}{z^j}\right]=0.
\]
There is an initial term equal to $(n^r+1)f_{\vec m,(\vec n,r),j,s}$ and all other terms are determined by induction and $f(u,z)$. This method recursively determines $F(t,u,z)$ from $F(t,u=0,z)$.

Setting $\hat\bt(z)=\mathbf{t}(z)+f(u,z)$, we conclude that
\[
F(t,u,z)=z\phi_0+\hat\bt(-z)+\sum_{k}\phi^k\left\langle\left\langle \frac{\phi_k}{z-\psi_{1}}\right\rangle\right\rangle_{1}^{\infty}\left(\hat t\right)
\]
because both sides agree along the restriction $u=0$, they have the same regular part, and they both satisfy the same recursion which determines them uniquely from this initial data.  Therefore, $F(t,u,-z)\in\cL$.
\end{proof}

Theorem \ref{thm} now follows immediately from Lemmas \ref{lem:relation1} and \ref{thm:determine}.

\section{Applications}\label{sec:mirror}

We now extract several corollaries from Theorem \ref{thm}. Define $J_0^\epsilon(u)\in \Q[[u]]$ and $J_1^\epsilon(u,z)\in \cH_+[[u]]$ by
\[
\cJ^\epsilon(t,u,z)=J_0^\epsilon(u)\phi_0z+\bt(z)+J_1^\epsilon(u,z)+O\left(\frac{1}{z}\right).
\]
Define the change of variable $\btau^\epsilon(t,u,z)$ by
\[
\btau^\epsilon(t,u,z):=J_0^\epsilon(u)\bt(z)-J_1^\epsilon(u,z)
\]
and write $\tau^\epsilon=(\tau^\epsilon(t,u)_j^k)$. We have the following comparison of generating series.

\begin{theorem}\label{cor1}
\begin{equation}\label{jfxnmatch}
\frac{\cJ^{\epsilon_1}(\tau^{\epsilon_1},u,z)}{J_0^{\epsilon_1}(u)} = \frac{\cJ^{\epsilon_2}(\tau^{\epsilon_2},u,z)}{J_0^{\epsilon_2}(u)}
\end{equation}
\end{theorem}

\begin{proof}
Notice that $J_0^\epsilon(u)=1+O(u)$, so both sides make sense as elements of $\cH[[u]]$. Since $\cL$ is a cone, both sides of \eqref{jfxnmatch} lie on $\cL$ by Theorem \ref{thm} (upon negating $z$). The change of variables is defined precisely so that the two sides of \eqref{jfxnmatch} agree in they regular parts. The result now follows from the simple observation that $\cL$ is a graph over $\cH^+$, meaning that points on $\cL$ are uniquely determined by their regular part.
\end{proof}

We now turn our attention to the big $I$ function 
\[
\bI(u,z)=z\phi_0\sum_{a_i\geq 0}\prod_i\frac{1}{a_i!}\left(\frac{u^i\phi_i}{z}\right)^{a_i}\prod_{j=1}^N\prod_{\substack{0\leq b<\sum_ia_i\langle iq_j\rangle\\ \langle b \rangle =\langle \sum_i a_i iq_j\rangle}}(b+q_j)z.
\] 
Since $\bI(u,z)$ has unbounded positive powers of $z$, we need to work with a completion of our base ring $\Q[[u]]$. In particular, we work over the $u$-adic completion of $\Q[[u]]$ defined in terms of the total degree in the $u^i$ where $\deg(u^i)=i$. We take $\cH$ to contain series $\sum_{j\in\Z} h_jz^j$ which are possibly infinite in both directions, but we require that $\lim_{j\rightarrow\infty}h_j\rightarrow 0$. With this convention, Theorem \ref{thm} implies that $\bI(u,-z)$ lies on $\cL$.

One of the fundamental results in Givental's approach to axiomatic GW theories \cite{g:sgofs} is that the cone $\cL$ is swept by a finite-dimensional family of semi-infinite linear spaces. In particular, if $T$ is a tangent space of $\cL$, then $T$ is tangent to $\cL$ exactly along $zT$. Moreover, the cone is completely determined by any transverse slice of dimensions $\dim H_W'$; in particular,
\[
\cL=\bigsqcup_{\bt_0\in H'} zT_{J(t_0,-z)}\cL. 
\]
where the big $J$-function $J(t_0,z)$ is defined by restricting $\cJ^\infty(t,z)$ to primary variables $t_0=(t_0^k)$. Consequently, $\bI(u,-z)\in zT_{J(\sigma(u),-z)}$ where $\bm\sigma(u)\in H'$ is the unique point in the intersection 
\[
zT_{\bI(u,-z)}\cL\cap (-z\phi_0+z\cH^-)\subset \cL\cap(-z\phi_0+z\cH^-).
\]
The map $u\rightarrow \bm\sigma(u)$ has the following important property.

\begin{lemma}\label{determinecone}
Let $D$ be the dimension of $H'$. Then $\bm\sigma$ defines an isomorphism $\bm\sigma:\Q^D\rightarrow H'$ when restricted to a formal neighborhood of the origin.
\end{lemma}

\begin{proof}
By definition,
\[
\bI(u,-z)=-z\phi_0+\bu+O(u^2).
\]
Therefore, $\bm\sigma(u)=\bu+O(u^2)$ implying that $\bm\sigma(u)$ is invertible over $\Q[[u]]$.
\end{proof}

In particular, since $J(t,-z)$ determines the entire cone $\cL$, Lemma \ref{determinecone} implies the mirror theorem, proved previously by Chiodo--Iritani--Ruan for $q=1$ and independently by Acosta for $q\neq1$.

\begin{theorem}[\cite{cir:lgcy,a:thesis}]\label{cor2}
$\bI(u,-z)$ determines the entire cone $\cL$.
\end{theorem}

In fact, one can recursively invert $\bm\sigma(u)$ for low powers of $u$ and thus determine the FJRW invariants purely in terms of $\bI$. For examples of how this process is carried out explicitly, see \cite{ccit:saotmt}.

A simple calculation shows that the power of $z$ which appears in the $\prod_i(u^i)^{a_i}$ coefficient of the big $I$-function is at most
\[
1+\sum_i a_i(\deg(\phi_i)-1)
\]
Therefore, since the small $I$-function is defined by restricting to degree $\leq 1$ insertions, it has the form
\[
I(\underline{u},z)=I_0(\underline{u})z\phi_0+I_1(\underline{u})+O\left(\frac{1}{z}\right)
\]
for some $I_0(\underline{u})=1+O(\underline{u})\in\Q[[\underline{u}]]$ and $I_1(u)\in H'[[\underline{u}]]$. We recover the usual LG mirror theorem which relates the small $I$-function to the $J$-function via a simple change of variables.

\begin{corollary}\label{cor3}
\[
\frac{I(\underline{u},z)}{I_0(\underline{u})}=J(\eta(\underline{u}),z).
\]
where $\bm\eta(\underline{u})=\frac{I_1(\underline{u})}{I_0(\underline{u})}$.
\end{corollary}

\begin{proof}
Both sides lie on $\cL$ (after negating $z$) and they agree in their regular part.
\end{proof}

\begin{remark}
We note that the statement of Corollary \ref{cor3} is only interesting for $q\leq 1$. For $q>1$, it is easily checked that $J=I=z\phi_0$.\end{remark}

As a last  consequence of our wall-crossing formula, we derive an enumerative description of Corollary \ref{cor3}.

\begin{theorem}\label{cor4}
\[
\frac{I(\underline{u},z)}{I_0(\underline{u})}=z\phi_0+z\sum_k \phi^k\left\langle\left\langle \phi_0,\frac{\phi_k}{z-\psi_2}\right\rangle\right\rangle^0_2(0,\underline u)
\]
and
\[
\bm\eta(\underline{u})=\sum_k \phi^k\left\langle\left\langle \phi_0,\phi_k\right\rangle\right\rangle^0_2(0,\underline u)
\]
\end{theorem}

\begin{proof}
By setting $t=0$ in Theorem \ref{cor1} and considering the extreme values for $\epsilon$, we compute 
\begin{align}\label{enumerative}
\nonumber \frac{I(\underline{u},z)}{I_0(\underline{u})}&=\cJ^\infty\left(\frac{I_1(\underline{u})}{I_0(\underline{u})},z\right)\\
&\nonumber=z\phi_0+\frac{I_1(\underline{u})}{I_0(\underline{u})}+\sum_{m,k}\frac{\phi^k}{m!}\left\langle \left(\frac{I_1(\underline{u})}{I_0(\underline{u})}\right)^m,\frac{\phi_k}{z-\psi_{m+1}}  \right\rangle_{m+1}^\infty\\
&=z\phi_0+z\sum_{m,k}\frac{\phi^k}{m!}\left\langle\phi_0, \left(\frac{I_1(\underline{u})}{I_0(\underline{u})}\right)^m,\frac{\phi_k}{z-\psi_{m+1}}  \right\rangle_{m+1}^\infty
\end{align}
where the last equality follows from the string equation in FJRW theory. The final expression \eqref{enumerative} is equal to the coefficient of $t_0$ in $zI_0(\underline{u})\cJ^\infty\left(\tau^0(t,\underline u)^{-1},z \right)$ and by Theorem \ref{cor1}, we have
\[
zI_0(\underline{u})\cJ^\infty\left(\tau^0(t,\underline u)^{-1},z \right)=z\cJ^0(t,\underline{u},z).
\]
The theorem follows by extracting the $t_0$ coefficient of $z\cJ^0(t,\underline{u},z)$.
\end{proof}

For all $\epsilon>0$ we can define the J-function by
\[
\text{J}^\epsilon(\underline u,z):=z\phi_0+z\sum_k \phi^k\left\langle\left\langle \phi_0,\frac{\phi_k}{z-\psi_2}\right\rangle\right\rangle^\epsilon_2(0,\underline u).
\]
By the string equation, $\text{J}^1(\underline u, z)$ is nothing more than the small FJRW $J$-function. Theorem \ref{cor4} then asserts that the usual mirror theorem in Corollary \ref{cor3} is obtained by tracking the generating series $\text{J}^\epsilon(\underline{u},z)$ as we vary $\epsilon$ from $1$ to $0$. In other words, up to a correction of $\phi_0$ by $I_0(\underline u)$, the $I$-function is equal to the J-function at $\epsilon=0$ without any change of variables.

\bibliographystyle{alpha}


\end{document}